\documentclass[12pt]{amsart}

\usepackage[initials]{amsrefs}
\usepackage{amssymb,amsmath,amsfonts,latexsym,amsthm,paralist,graphicx}

\usepackage[a4paper,nomarginpar]{geometry}

\usepackage[english]{babel}

\newtheorem{theorem}{Theorem}[section]

\newtheorem{proposition}[theorem]{Proposition}
\newtheorem{lemma}[theorem]{Lemma}

\theoremstyle{definition}

\input epsf

\newcommand{\N}{\mathbb{N}} 
\newcommand{\Z}{\mathbb{Z}} 
\newcommand{\R}{\mathbb{R}} 
\newcommand{\C}{\mathbb{C}} 
\newcommand{\D}{\mathbb{D}} 
\newcommand{\Hol}{\mathcal{H}(\mathbb{C})} 
\newcommand{\ve}{\varepsilon}

\begin{document}

\title[Hypercyclic functions for convolution operatos]{Multiplicative structures of hypercyclic functions for convolution operators}

\date{}

\author[Bernal]{Luis Bernal-Gonz\'alez}
\address{Departamento de An\'{a}lisis Matem\'{a}tico,
\newline\indent Facultad de Matem\'{a}ticas, Universidad de Sevilla,
\newline\indent Avenida Reina Mercedes,
\newline\indent 41080 Sevilla, Spain.}
\email{lbernal@us.es}

\author[Conejero]{J. Alberto Conejero}
\address{Instituto Universitario de Matem\'{a}tica Pura y Aplicada,
\newline\indent Universitat Polit\`{e}cnica de Val\`{e}ncia,
\newline\indent 46022 Val\`{e}ncia, Spain.}
\email{aconejero@upv.es}

\author[Costakis]{George Costakis}
\address{Department of Mathematics and Applied Mathematics, \newline\indent
University of Crete, Voutes Campus,
\newline\indent 70013 Heraklion, Crete, Greece.}
\email{costakis@math.uoc.gr}

\author[Seoane]{Juan B. Seoane-Sep\'{u}lveda}
\address{IMI and Departamento de An\'{a}lisis Matem\'{a}tico,\newline\indent Facultad de Ciencias Matem\'{a}ticas,
\newline\indent Plaza de Ciencias 3,
\newline\indent Universidad Complutense de Madrid,\newline\indent 28040 Madrid, Spain.}
\email{jseoane@mat.ucm.es}


\dedicatory{Dedicated to Professor Domingo Garc\'{\i}a on the occasion of his 60th birthday}

\thanks{}

\begin{abstract}
In this note, it is proved the existence of an infinitely generated multiplicative group consisting of entire functions that are,
except for the cons\-tant function $1$, hypercyclic with respect to the convolution operator associated to a given entire function of subexponential type.
A certain stability under multiplication is also shown for compositional hypercyclicity on complex domains.
\end{abstract}

\maketitle

\section{Introduction}

Assume that $X$ is a (Hausdorff) topological vector space over the real line $\R$ or the complex plane $\C$,
and consider the vector space $L(X)$ of all operators on $X$, that is, the family of
all continuous linear self-mappings $T:X \to X$. An operator $T \in L(X)$ is said to be {\it hypercyclic} provided that it admits a dense orbit, that is,
provided that there is some vector $x_0 \in X$ (called hypercyclic for $T$) such that the orbit $\{T^n x_0: \, n \in \N\}$ of $x_0$ under $T$ is dense in $X$
($\N := \{1,2, \dots \}$). The set of hypercyclic vectors for $T$ will be denoted by $HC(T)$. In this paper, we are concerned with the {\it size} of $HC(T)$, mainly from an algebraic point of view and for certain differential operators. For background on hypercyclic operators we refer the reader to the excellent books \cites{bayartmatheron2009,grosseperis2011}. An account of concepts and results about algebraic structures inside nonlinear sets can be found in \cites{arongurariyseoane2005,aronbernalpellegrinoseoane,bernalpellegrinoseoane2014linear,TAMS-Enflo,Seo1,Seo2,Seo3}.

\vskip .15cm

It is well known that if $X$ is an F-space (i.e., complete and metrizable) and $T$ is a hypercyclic operator then the set $HC(T)$ is {\it residual,} that is, it contains a dense $G_\delta$ subset of the (Baire) space $X$; we can say that $HC(T)$ is topologically large. Furthermore, for any topological vector
space and any hypercyclic $T \in L(X)$, the family $HC(T)$ is algebraically large; specifically, it contains, except for \,$0$, a {\it dense} (even $T$-invariant) vector subspace of \,$X$ (see \cite{wengenroth2003}). Starting from \cite{montes1996}, a number of criteria have been established for an operator $T \in L(X)$ to support a {\it closed infinite dimensional} vector subspace all of whose nonzero vectors are $T$-hypercyclic on an F-space $X$; however, not all hypercyclic operators support such a subspace  (see \cite[Chapter 8]{bayartmatheron2009} and \cite[Chapter 10]{grosseperis2011}).

\vskip .15cm

By a domain in \,$\C$ we mean a nonempty connected open subset \,$G \subset \C$. It is well known that the space \,$\mathcal H (G)$ \,of holomorphic functions
\,$G \to \C$ \,becomes an F-space when endowed with the topology of uniform convergence on compact subsets of \,$G$.
We are mainly interested in operators defined on the space $\Hol$ \,of entire functions \,$\C \to \C$.
An operator $T \in L(\Hol )$ is said to be a {\it convolution operator} \,provided that it commutes with translations, that is,
$$
T \circ \tau_a = \tau_a \circ T  \hbox{ \ for all \ } a \in \C ,
$$
where $\tau_a f := f(\cdot + a)$. Then $T \in L(\Hol )$ happens to be a convolution operator if and only if $T$ is an infinite order linear differential operator with constant coefficients $T = \Phi (D)$, where $\Phi$ is an entire function with exponential type, that is, there exist
constants $A,B \in (0,+\infty )$ such that $|\Phi (z)| \le A e^{B|z|}$ for all $z \in \C$. For such a function $\Phi (z) = \sum_{n \ge 0} a_nz^n$ and $f \in \Hol$, we have
$$
\Phi (D) f = \sum_{n = 0}^\infty a_n f^{(n)} .
$$
Godefroy and Shapiro \cite{godefroyshapiro1991} proved that every non-scalar convolution operator is hypercyclic, so covering the classical Birkhoff and MacLane results on hypercyclicity of the translation operator $\tau_a$ (take $\Phi (z) = e^{az}$, $a \ne 0$) and of the derivative operator $D: f \mapsto f'$ (take $\Phi (z) = z$), respectively. It has been proved that, for
every non-scalar convolution operator $T$, the set \,$HC(T)$ \,contains, except for $0$, a closed infinite dimensional vector subspace of $\Hol$; see
\cites{menet2014,petersson2006,shkarin2010} and also \cite[Section 4.5]{aronbernalpellegrinoseoane} and \cite[Section 10.1]{grosseperis2011}.

\vskip .15cm

In view of the preceding paragraph, we can say that hypercyclic vectors of convolution operators are rather stable under summation and scaling.
However, stability {\it under multiplication} does not seem to be so clear. For instance, for every $a \ne 0$, every $f \in HC(\tau_a)$ and every
$n \in \N$ with $n \ge 2$ we have that $f^n \not\in HC(\tau_a)$ (see \cite[Cor.~2.4]{aronconejeroperisseoane2007}).
As a positive result, it was proved in \cite[Th.~2.3]{aronconejeroperisseoane2007} the existence of a function $f \in \Hol$
--in fact, of a residual subset of them-- such that \,$f^n \in HC(D)$ \,for all $n\in\N$, see also \cite{aronconejeroperisseoane2007sums}.
This result was extended by Bayart and Matheron who proved that there is even a residual set of functions in \,$\Hol$
\,generating a hypercyclic algebra for the derivative operator, that is every non-null function in one of these algebras
is hypercyclic for the operator \,$D$ \,\cite[Th.~8.26]{bayartmatheron2009}. Recall that the algebra generated by a function $f$ is nothing but
the set \,$\{P \circ f: \, P$ polynomial, $P(0) = 0\}$.
The existence of (one-generated) algebras of hypercyclic functions for \,$D$ \,was also independently proved by Shkarin in \cite{shkarin2010}
and extended by B\`es {\it et al.}~in \cite{besconejeropapathaniasiou2016} to operators of the form $P(D)$, where $P$ is a nonconstant polynomial with $P(0) = 0$. The existence of hypercyclic algebras for many convolution operators not induced by polynomials, such as $\cos(D)$, $De^D$, or $e^D-aI$, where $0<a\le 1$ was recently shown in \cite{besconejeropapathaniasiou2018}.

\vskip .15cm

So far, the existence of multiply generated algebras of hypercyclic entire functions is unknown.
The lack of knowledge about existence of families of hypercyclic functions that are stable under multiplication seems to be one
of the main obs\-ta\-cles. Our goal in this paper is to contribute to fill in this gap.
Specifically, given a nonconstant entire function \,$\Phi$ \,of subexponential type, we will construct in Section 3 an infinitely generated multiplicative group of entire functions all of whose members, except for the constant function ${\bf 1} (z) = 1$, are hypercyclic with respect to the operator $\Phi (D)$.
A certain multiplicative stability is also shown for hy\-per\-cy\-cli\-ci\-ty with respect to composition operators. Section 2 is devoted to provide
the necessary background and auxiliary results.

\section{Non-vanishing hypercyclic entire functions}

Recall that the exponential type of a $\Phi \in \Hol$ is defined as, see \cite{boas1954},
$$
\tau (\Phi ) = \limsup_{r \to \infty} {\log \max \{|\Phi (z)|: \, |z|=r\} \over r}.
$$
Hence, $f$ has subexponential type if and only if, given $\ve > 0$, there exists $A = A(\ve ) \in (0,+\infty )$ such that
$|f(z)| \le A \, e^{\ve |z|}$ for all $z \in \C$. In fact, if \,$\Phi$ \,has subexponential type then \,$\Phi (D)$ \,is
a well defined operator on \,$\mathcal{H} (G)$ \,for any domain $G \subset \C$  (see for instance \cite{berensteingay1995} or \cite[Theorem 4]{bernal1999}).
It should be underlined that if \,$\Phi$ \,is of exponential type, then the
operator \,$\Phi (D)$ \,does not act on proper domains in general: take, for instance, $G = \D$ \,and \,$\Phi (z) = e^z$.

\vskip .15cm

As we have mentioned in Section 1, there is no entire function $f$ such that $f^2$ is hypercyclic for the translation operator $\tau_a = e^{aD}$ $(a \ne 0)$.
If we observe that \,$\Phi (z) := e^{az}$ \,has exponential type \,$\tau (\Phi ) = |a| > 0$, then we may wonder whether there are
entire functions $f$ such that $f$ and $f^2$ are in $HC(\Phi (D))$ if $\Phi$ has {\it subexponential type,} i.e., if \,$\tau (\Phi ) = 0$.
This property happens to be true, in a strong sense, in the special case of \,$\Phi = P$, a nonconstant polynomial with
\,$P(0) = 0$ \cite{besconejeropapathaniasiou2016}. The rate of growth of a hypercyclic function for $D$ has been optimally estimated in
\cites{grosse1990universal,shkarin1993}.

\vskip .15cm

It is well known (see, for instance, \cite{ahlfors1979}) that a function $f \in \Hol$ is never zero if and only if
there is $g \in \Hol$ such that $f = e^g$. Consistently, we can denote the family of non-vanishing entire functions
by \,$e^{\Hol}$. This family forms, trivially, a group under the pointwise multiplication. Moreover, $e^{\Hol}$ \,is --as it is easily proven--
a topological group (for background on topological groups, see for instance \cite{markley2010}) under the topology of uniform convergence in compacta, that is, both mappings
$$
(f,g) \in e^{\Hol} \times e^{\Hol} \mapsto fg \in e^{\Hol} \hbox{ \ and \ } f \in e^{\Hol} \mapsto {1 \over f} \in e^{\Hol}
$$
are continuous,
from which one can derive that each mapping \,$f \in e^{\Hol} \mapsto f^m \in e^{\Hol}$ ($m \in \Z =$ the set of integers) \,is
a continuous homomorphism (recall that a mapping \,$\Phi : G_1 \to G_2$ \,between two groups \,$G_1,G_2$ \,is an homomorphism whenever
\,$\Phi (ab) = \Phi(a) \Phi (b)$ \,for all \,$a,b \in G_1$).
Note that this mapping is not bijective (except for $m = \pm 1$), but if we
consider the topological subgroup \,$e^{\Hol}_{+,0} := \{f \in e^{\Hol}: \, f(0) > 0\}$ \,then each restriction
$$
\Phi_m: f \in e^{\Hol}_{+,0} \mapsto f^m \in e^{\Hol}_{+,0} \quad (m \in \Z \setminus \{0\})
$$
is not only continuous (and homomorphic) but also bijective, because for any $g \in e^{\Hol}_{+,0}$ there are exactly
\,$m$ \,entire functions \,$f$ \,with \,$f^m = g$ (necessarily, such $f$'s are non-vanishing), but only one of them satisfies
$f(0) > 0$. In fact, in the following two lemmas it is shown that the group \,$e^{\Hol}_{+,0}$ \,also enjoys a nice topological structure
and that every \,$\Phi_m$ \,is an automorphism of it.

\begin{lemma} \label{Lemma: topology of e+0}
The set \,$e^{\Hol}_{+,0}$ \,is a $G_\delta$-subset of \,$\Hol$. In particular, it is a separable completely metrizable space, as well as a Baire space.
\end{lemma}

\begin{proof}
The second part is a consequence of the first one. Indeed,  Alexandroff's theorem (see e.g.~\cite[pp.~47--48]{oxtoby1980}) asserts
that a $G_\delta$-set is homeomorphic to a complete metrizable space; separability is inherited by any subspace of a separable metrizable space;
finally, any completely metrizable space is a Baire space (see e.g.~\cite{munkres2000}). As for the first part, simply put
\,$e^{\Hol}_{+,0} = \bigcap_{k=1}^\infty A_k$, where \,$A_k$ \,is defined as
$$
A_k := \left\{f\in\Hol :\, \inf_{|z|\leq k}|f(z)|>0, \, |{\rm Im} f(0)| < {1 \over k} \text{ and }\, {\rm Re} f(0) > 0\right\}.
$$
That each \,$A_k$ \,is open follows from the facts that \,$\{z\in\mathbb{C}\,:\,|z| \le k\}$ \,is compact and that both
evaluation mappings \,$f \in \Hol \mapsto {\rm Re} f(0) \in \R$,
$f \in \Hol \mapsto {\rm Im} f(0) \in \R$ \,are continuous.
\end{proof}

\begin{lemma} \label{Lemma: onto homeomorphisms}
For each \,$m \in \Z \setminus \{0\}$, the mapping \,$\Phi_m$ \,is an onto homeomorphism.
\end{lemma}

\begin{proof}
The unique property to be proved is that every \,$\Phi_m$ \,has inverse continuous. Since \,$\Phi_{-1}$ \,is an onto
homeomorphism, it is enough to see that,
given \,$m \in \N$ \,with \,$m \ge 2$, the mapping \,$\Psi = (\Phi_m)^{-1}: e^{\Hol}_{+,0} \to e^{\Hol}_{+,0}$ \,is continuous.
By Lemma \ref{Lemma: topology of e+0}, $e^{\Hol}_{+,0}$ is a completely metrizable topological group.
Note that \,$\Psi$ \,is a homomorphism from the group \,$e^{\Hol}_{+,0}$ \,into itself.
According to the abstract closed graph theorem (see, e.g., \cite[Theorem 5.2]{christensen}), it is enough to
show that \,$\Psi$ \,has closed graph. To this end, assume that \,$(f_k) \subset e^{\Hol}_{+,0}$ \,is a sequence
such that there are \,$f,g \in e^{\Hol}_{+,0}$ \,with \,$(f_k, \Psi (f_k)) \to (f,g)$ \,as \,$k \to \infty$ \,in the
topology product. Then \,$f_k \to f$ \,and \,$\Psi (f_k) \to g$.
The continuity of \,$\Phi_m$ \,yields \,$f_k = \Phi_m \Psi (f_k) \to \Phi_m (g)$, and the uniqueness of the limit implies
\,$f = \Phi_m (g)$ \,or, that is the same, $g = \Psi (f)$, which proves that the graph of \,$\Psi$ \,is closed, as required.
\end{proof}

Finally, we will use in the next section the following theorem, that is contained in \cite[Theorem 5]{bernal1997}.
This theorem is, in turn, an extension of a result about inherited $D$-hypercyclicity due to Herzog \cite{herzog1994}.

\begin{theorem} \label{Th: Bernal 1997}
If \,$\Phi$ \,is a nonconstant entire function of subexponential type, then the set \,$HC(\Phi (D)) \cap e^{\Hol}$ \,is residual in \,$e^{\Hol}$.
\end{theorem}

Additional statements on zero-free $D$-hypercyclic entire functions can be found in \cite{bernalbonillacostakis2012}.

\section{Groups of hypercyclic functions}

\quad Firstly, we show that the family of entire functions $f$ that are hypercyclic with respect to certain convolution operators (including $D$)
satisfying $f(0) > 0$ is topologically large.

\begin{proposition} \label{Prop: f(0) > 0 residual}
If \,$\Phi$ \,is a nonconstant entire function of subexponential type, then the set \,$HC(\Phi (D)) \cap e^{\Hol}_{+,0}$ \,is residual in \,$e^{\Hol}_{+,0}$.
\end{proposition}
\begin{proof}
According to Birkhoff transitivity theorem (see e.g.~\cite{grosseperis2011}), if $T_n : X \to Y$ $(n \in \N )$ is a sequence of continuous mappings between
two Hausdorff topological spaces, with $X$ Baire and first-countable and $Y$ second-countable, then the set \,$\mathcal U = \mathcal U ((T_n))$ \,of points $x \in X$ whose orbit \,$\{T_n x : \, n \in \N\}$ \,is dense in $Y$ is a $G_\delta$ subset of $X$; consequently, $\mathcal U$ \,is residual as soon as it is dense.
By Theorem \ref{Th: Bernal 1997}, the set \,$HC(\Phi (D)) \cap e^{\Hol}$ \,is residual in \,$e^{\Hol}$. It is easy to see that
the mapping
$$
\Pi : f \in e^{\Hol} \mapsto {|f(0)| \over f(0)} \cdot f \in  e^{\Hol}_{+,0}
$$
is continuous and onto. Then it takes dense sets into dense sets.
In particular, the image \,$\Pi (HC(\Phi (D)) \cap e^{\Hol})$ \,is dense in \,$e^{\Hol}_{+,0}$. Since any nonzero scalar multiple of a hypercyclic vector is also hypercyclic, we obtain \,$\Pi (HC(\Phi (D)) \cap e^{\Hol}) \subset HC(\Phi (D)) \cap e^{\Hol}_{+,0}$. But \,$\mathcal U := HC(\Phi (D)) \cap e^{\Hol}_{+,0}$
equals the set of elements \,$f \in X := e^{\Hol}_{+,0}$ \,whose orbit under the sequence \,$(T_n) := ((\Phi (D))^n|_{e^{\Hol}_{+,0}})$ \,is dense
in \,$Y := \Hol$ (note that $X$ is Baire and first-countable by Lemma \ref{Lemma: topology of e+0}). Therefore $\mathcal U$ is a $G_\delta$ dense subset of
\,$e^{\Hol}_{+,0}$, hence residual.
\end{proof}

Our main result (Theorem \ref{Th: groups in HC(Phi(D))}) will be deduced as a consequence of the following, more abstract, assertion.

\begin{theorem} \label{Th: groups in residual subsets}
Let \,$X$ be a separable infinite dimensional F-space. Assume that $X$ is also a unitary commutative linear algebra, and that the corresponding multiplication law $(f,g) \in X \times X \mapsto f * g \in X$ is continuous. Suppose also that \,$Z$ is a subset of \,$X$ fulfilling the following
conditions:
\begin{enumerate}
\item[\rm (i)] $Z$ is a $G_\delta$-subset of $X$.
\item[\rm (ii)] $(Z,*)$ is a topological group.
\item[\rm (iii)] The mappings $u \in Z \mapsto u^k \in Z$ $(k \in \N )$ are onto homeomorphisms.
\item[\rm (iv)] No nonempty relatively open subset of \,$Z$ is contained in a countable dimensional vector subspace of \,$X$.
\end{enumerate}
Then for each residual subset \,$\mathcal R$ of \,$Z$ there exists a subset \,$\mathcal G \subset Z$ satisfying the following properties:
\begin{enumerate}
\item[\rm (a)] $\mathcal G$ is a subgroup of \,$Z$.
\item[\rm (b)] $\mathcal G$ is dense in \,$Z$.
\item[\rm (c)] $\mathcal R \supset \mathcal G \setminus \{{\bf e}\}$, where \,${\bf e}$ \,denotes the unit element of \,$X$.
\item[\rm (d)] $\mathcal G$ is infinitely generated in a strong sense, namely, the algebra generated by \,$\mathcal G$
is infinitely generated.
\end{enumerate}
\end{theorem}

\begin{proof}
By Alexandroff's theorem and by the fact that $X$ is a separable F-space, we have (thanks to (i)) that $Z$ is a completely metrizable space that,
in addition, is second-countable. Then $Z$ is a Baire space and there is a countable open basis
\,$\{G_n\}_{n \ge 1}$ for the topology of \,$Z$. From (ii) and (iii) one takes out that every mapping \,$\Psi_k : u \in Z \mapsto u^k \in Z$
$(k \in \Z \setminus \{0\})$ \,is an onto homeomorphism. Moreover, it follows from (ii) that, for every \,$v \in \Z$,
the multiplication mapping \,$M_{v} : u \in Z \mapsto v * u \in Z$ \,is also an onto homeomorphism.

\vskip 5pt

Fix a residual subset \,$\mathcal R$ of \,$Z$. Then, for each $k \in \Z \setminus \{0\}$,
the set \,$\Psi_k^{-1} (\mathcal R)$ \,is residual in $Z$. Then the countable intersection
$$
\mathcal R_0 := \bigcap_{k \in \Z \setminus \{0\}} \Psi_k^{-1} (\mathcal R)
$$
is also residual --and, in particular, dense-- in \,$Z$. Let us set $L_0 := {\rm span} \{{\bf e}\}$, which equals the algebra generated by ${\bf e}$.
Since $L_0$ is a finite dimensional subspace, it is closed in $X$. Then $G_1 \setminus L_0$ is a nonempty
(otherwise, $G_1 \subset L_0$, contradicting (iv)) open set in the topology of \,$Z$. Then there exists $v_1 \in G_1 \cap \mathcal R_0$.

\vskip 5pt

Let us proceed by induction. Assume that, for some $n \in \N$, the vectors $v_1, \dots ,v_n$ have been selected.
Let $L_n$ denote the linear algebra generated by $v_0,v_1, \dots ,v_n$, that is, $L_n = {\rm span} \{v_1^{j_n} * \cdots * v_n^{j_n}: \, (j_1, \dots ,j_n) \in \N_0^n\}$, where $\N_0 := \N \cup \{0\}$. Observe that \,$L_n$ \,is a countable dimensional vector subspace of $X$. Since $X$ is a Baire space, if follows that $L_n$ is an $F_\sigma$ set with empty interior. Now, we set
$$
\mathcal R_n := \bigcap_{k \in \Z \setminus \{0\}} \Psi_k^{-1} \bigg( \bigcap_{(k_1, \dots ,k_n) \in \Z^n}
M^{-1}_{v_1^{k_1} * \cdots * v_n^{k_n}} (\mathcal R) \bigg) .
$$
Observe that the choice $(k_1, \dots ,k_n) = (0, \dots ,0)$ gives $\mathcal R_n \subset \mathcal R_0$.
Since the $\Psi_k$'s as well as the $M^{-1}_{v_1^{k_1} * \cdots * v_n^{k_n}}$'s are onto homeomorphisms and countable intersections
of residual sets are residual, we get that each \,$\mathcal R_n$ \,is residual in \,$Z$. Assume, by way of contradiction, that
$G_{n+1} \setminus L_n$ is of first category (in the sense of Baire) in $Z$. Then $G_{n+1} = (G_{n+1} \setminus L_n) \cup (G_{n+1} \cap L_n)$ and, by the assumption (iv), $L_n \cap Z$ is a (relatively $F_\sigma$) set with empty interior in $Z$, so of first category in $Z$. This implies that its subset $G_{n+1} \cap L_n$ is also of first category in $Z$, hence $G_{n+1}$ is, which is absurd because $Z$ is a Baire space. Thus,
$G_{n+1} \setminus L_n$ is of second category (in the sense of Baire) in $Z$. As $\mathcal R_n$ is residual in $Z$, we have
$\mathcal R_n \cap (G_{n+1} \setminus L_n) \ne \varnothing$. Therefore, we can select $v_{n+1} \in \mathcal R_n \cap (G_{n+1} \setminus L_n)$ and the induction procedure is finished.

\vskip 5pt

Define \,$\mathcal G$ \,as the group generated by \,$\{v_n\}_{n \in \N}$, that is
$$
\mathcal G = \big\{ v_1^{k_1} * \cdots * v_n^{k_n} : \, (k_1, \dots ,k_n) \in \Z^n, \, n \in \N \big\} .
$$
Since each $v_n$ belongs to $Z$ and $v_n \in G_n$ $(n \in \N )$, we obtain conclusions (a) and (b). Property (c) is derived from the fact that, by
construction, each combination $v_1^{k_1} * \cdots * v_n^{k_n} \in \mathcal R$ as soon as $(k_1, \dots ,k_n) \in \Z^n \setminus \{(0, \dots ,0)\}$.
Finally, property (d) is true because the algebra generated by $\mathcal G$ contains the algebra $\mathcal A := \bigcup_{n \ge 0} L_n$, and $\mathcal A$ is infinitely generated. Indeed, assume that $\mathcal A$ is generated by some finite subset $F \subset \mathcal A$. Since $L_n \subset L_{n+1}$ $(n \ge 0)$, there is $N \in \N$ such that $F \subset L_N$. Then $L_N$ generates $\mathcal A$. But $L_N$ is itself an algebra, so $\mathcal A = L_N$, which is absurd because $v_{N+1} \in \mathcal A \setminus L_N$. Hence $\mathcal A$ cannot be finitely generated, which concludes the proof.
\end{proof}

\begin{theorem} \label{Th: groups in HC(Phi(D))}
Assume that \,$\Phi$ \,is a nonconstant entire function of subexponential type.
Consider the convolution operator \,$T = \Phi (D)$ \,associated to \,$\Phi$.
Then there exists an infinitely generated multiplicative group \,$\mathcal G \subset \mathcal H (\C )$
\,such that \,every member of \,$\mathcal G$ --except for {\bf 1}--
is $\Phi (D)$-hypercyclic. Moreover, $\mathcal G$ \,is a
dense subgroup of \,$e^{\Hol}_{+,0}$ \,and the algebra generated by \,$\mathcal G$ \,is infinitely generated.
\end{theorem}

\begin{proof}
It suffices to apply Theorem \ref{Th: groups in residual subsets}
with the following ``characters'': Take $X := \Hol$, $Z := e^{\Hol}_{+,0}$, $* :=$ the
usual multiplication, ${\bf e} := {\bf 1}$ \,and \,$\mathcal R := HC(\Phi (D)) \cap e^{\Hol}_{+,0}$.
Now, $Z$ is a topological group for which the mappings $f \mapsto f^k$ $(k \in \N )$ are onto
homeomorphisms by Lemma \ref{Lemma: onto homeomorphisms}, $Z$ is a $G_\delta$ subset of $X$ by Lemma \ref{Lemma: topology of e+0}, and \,$\mathcal R$ \,is residual
in \,$Z$ \,by Proposition \ref{Prop: f(0) > 0 residual}.

\vskip 5pt

Now, assume that \,$O$ \,is a nonempty open subset in \,$e^{\Hol}_{+,0}$ \,and that, by way of contradiction,
\,$Y$ \,is a countable dimensional space of \,$\Hol$ \,with \,$O \subset Y$. Then \,${\rm span} \{h_n\}_{n \ge 1} = Y$
\,for certain \,$h_n \in \Hol$ $(n \ge 1)$. Setting \,$Y_n:= {\rm span} \{h_1, \dots ,h_n\}$ \,we get \,$Y = \bigcup_{n \ge 1} Y_n$.
But each \,$Y_n$ \,is a finite dimensional subspace, hence $\sigma$-compact, and so \,$Y$ \,is. In other words, there
are countably many compact sets \,$K_n \subset e^{\Hol}_{+,0}$ $(n \in \N )$ \,such that \,$Y = \bigcup_{n \ge 1} K_n$.
Therefore each set
$Z_n := e^{\Hol}_{+,0} \cap K_n$ \,is compact (hence closed) in \,$e^{\Hol}_{+,0}$ and \,$O \subset \bigcup_{n \ge 1} Z_n$.
Since \,$e^{\Hol}_{+,0}$ \,is a completely metrizable space, Baire's category theorem
tells us that at least one \,$Z_m$ \,has nonempty interior in \,$e^{\Hol}_{+,0}$.
Now, the group structure of \,$e^{\Hol}_{+,0}$ entails that the function {\bf 1} possesses a compact neighborhood \,$W$ \,in $e^{\Hol}_{+,0}$.

\vskip 5pt

Now, consider the vector space \,$S := \{g \in \Hol : \, {\rm Im} \, g(0) = 0\}$ \,endowed with the topology inherited from $\Hol$.
Trivially, $S$ is a topological group for the operation ``+'' and the same topology.
For every $f \in e^{\Hol}_{+,0}$
there is a unique $g \in S$ such that $f = e^g$. Hence the mapping \,$T : f = e^g \in (e^{\Hol}_{+,0}, \cdot ) \mapsto g \in (S,+)$
\,is an algebraic group isomorphism. But it is in fact a topological isomorphism. Indeed, $T^{-1} : g \mapsto e^g$ is trivially continuous
(any superposition mapping $g \in \Hol \mapsto \varphi \circ g \in \Hol$, with $\varphi \in \Hol$, is continuous), and $T$ is continuous
because it is continuous at the neutral element {\bf 1}. This means that if a sequence $(f_n = e^{g_n}) \subset e^{\Hol}_{+,0}$
(with $g_n = u_n + iv_n \in S$, so that $v_n(0) = 0$ for all $n$) satisfies $f_n \to {\bf 1}$ then $g_n = u_n + iv_n = T(f_n) \to T({\bf 1}) = 0$
uniformly on compacta. This, in turn, follows by invoking the continuity of the principal
branch $\log_p w := {\rm ln} \, |w| + i \, {\rm arg}_p \, w$ of the logarithm on $\C \setminus (-\infty ,0]$.
Indeed, since $1 \in \C \setminus (-\infty ,0]$, for given $R > 0$ we have that $f_n (K) \subset \C \setminus (-\infty ,0]$
eventually, where $K = \{z: \, |z| \le R\}$. Then $\log_p (f_n) \to log_p 1 =0$ uniformly on $K$.
But \,$\log_p (f_n (z)) = u_n (z) + i \arg_p (e^{iv_n (z)}) = u_n (z) + i(v_n(z) + 2 k_n \pi )$
\,for certain $k_n \in \Z$ not depending on $z$. Hence $u_n \to 0$ and $v_n + 2k_n\pi  \to 0$ uniformly on $K$.
In particular, $2k_n\pi  = v_n(0) + 2k_n\pi \to 0$, so $k_n \to 0$ and, since $k_n \in \Z$, we get $k_n = 0$ eventually.
This implies $v_n \to 0$ uniformly on $K$, and so does $g_n = u_n + iv_n$ as $n \to \infty$, as required.
We have proved that \,$T$ \,is a homeomorphism. Thus \,$T(W)$ \,is a compact neighborhood of \,$0$ \,in \,$S$.
But Riesz's theorem (see, e.g., \cite[page 147, Theorem 3]{horvath1966}) implies that \,${\rm dim} \, (S) < \infty$, which is absurd
because \,$S$ \,contains all functions \,$z^n$ $(n \in \N )$.
This is the desired contradiction.

\vskip 5pt

Hence \,$O$ is not contained in any countable dimensional subspace of $\Hol$.
Consequently, conditions (i)--(iv) in Theorem \ref{Th: groups in residual subsets} are fulfilled and the conclusion follows.
\end{proof}

Another important class of operators is the one formed by the composition operators.
If $G \subset \C$ is a domain and $\varphi : G \to G$ is a holomorphic self-map of $G$, the mapping
$$
C_\varphi : f \in \mathcal (G) \mapsto f \circ \varphi \in H(G)
$$
is well defined, linear and continuous, and it is called
the {\it composition operator} \,with symbol \,$\varphi$.
By \,${\rm Aut} (G)$ \,we denote the group of automorphisms of \,$G$, i.e., the family of all bijective holomorphic functions \,$G \to G$.
The translations $\tau_a$ $(a \in \C )$ form a special instance when $G = \C$ and $\varphi (z) := z + a$.
We have already seen that, even for translations, the set of hypercyclic functions is not stable under products
($f^2$ is never $\tau_a$-hypercyclic, whatever $f$ is).
Nevertheless, we obtain a reasonably degree of stability if the products do not allow repetition of factors.
This can be done for many composition operators, as the following theorem --with which we conclude this paper-- shows.
Recall that a domain $G \subset \C$ is called simply connected if it lacks holes, that is, if $\C_\infty \setminus G$ is connected.

\begin{theorem}
Let \,$G \subset \C$ \,be a simply connected domain. Assume that \,$\varphi \in {\rm Aut} (G)$ \,and \,$\varphi$ \,has no fixed points.
Then for every \,$f \in HC(C_\varphi )$  
\,there is a family \,$\mathcal F \subset H(G)$ satisfying the following properties:
\begin{enumerate}
\item[\rm (a)] $f \in \mathcal F \subset HC(C_\varphi )$.
\item[\rm (b)] $\mathcal F$ is dense in $H(G)$.
\item[\rm (c)] For each finite nonempty subfamily \,$\mathcal F_0 \subset \mathcal F$, one has \,$\prod_{g \in \mathcal F_0} g \in HC(C_\varphi )$.
\item[\rm (d)] $\mathcal F$ is a linearly free system.
\end{enumerate}
\end{theorem}

\begin{proof}
Since $\varphi$ has no fixed points, the operator $C_\varphi$ is mixing (see \cite[Theorem 4.37]{grosseperis2011}).
In the context of F-spaces (as $H(G)$ is), this is equivalent to the fact that, for every (strictly increasing) subsequence $(n_k) \subset \N$,
the set \,$\mathcal U ((C_\varphi^{n_k}))$ \,of functions having dense $(C_\varphi^{n_k})$-orbit is dense (even residual) in $H(G)$. Of course, mixing property implies hypercyclicity.
Observe that $(C_\varphi)^m = C_{\varphi^m}$ for $m \in \N$, where $\varphi^1 := \varphi$ and $\varphi^{m+1} := \varphi \circ \varphi^m$.

\vskip 5pt

Fix $f \in HC(C_\varphi )$ as well as a countable topological basis $(G_n)$ for the topology of $H(G)$
(recall that $H(G)$ is metrizable and separable, hence second countable). We set $f_0 := f$ and $G_0 := H(G)$.
Then there is a sequence $(p(1,k)) \subset \N$
such that $f_0 \circ \varphi^{p(1,k)} \to 1$ $(k \to \infty )$ uniformly on each compact subset of $G$.
Since the set $\mathcal U ((C_\varphi^{p(1,k)}))$ is dense in $H(G)$, there exists $f_1 \in G_1 \cap \mathcal U ((C_\varphi^{p(1,k)}))$.
In particular, for the constant function $z \mapsto 1$, there is a subsequence $(p(2,k)) \subset (p(1,k))$ satisfying
$f_1 \circ \varphi^{p(2,k)} \to 1$ $(k \to \infty )$ in $H(G)$. Of course, $f_0 \circ \varphi^{p(2,k)} \to 1$ in $H(G)$.

\vskip 5pt

By following this procedure, we can obtain recursively a countable family \,$\mathcal F := \{f_n\}_{n \ge 0} \subset H(G)$
\,as a well as countable family \,$\{(p(n,k))_{k \ge 1}\}_{n \ge 0}$
\,of strictly increasing sequences of natural numbers, satisfying, for each $n \ge 0$, the following:
\begin{itemize}
\item $f_n \in G_n \cap \mathcal U ((C_\varphi^{p(n,k)}))$,
\item $(p(n+1,k)) \subset (p(n,k))$, and
\item $f_j \circ \varphi^{p(n+1,k)} \to 1$ $(k \to \infty )$ in $H(G)$ for all $j = 0, \dots ,n$.
\end{itemize}
Since $(G_n)$ is a topological basis and $f_n \in G_n$, the set $\mathcal F$ is dense in $H(G)$. By construction, $f \in \mathcal F$ and,
since $HC(C_\varphi ) \supset \mathcal U ((C_\varphi^{p(n,k)}))$ for every $n$, we get $\mathcal F \subset HC(C_\varphi )$.
Now, if $\varnothing \ne \mathcal F_0 \subset \mathcal F$ with \,$\mathcal F_0$ \,finite, we can write \,$\mathcal F_0 = \{f_{m(1)}, \dots , f_{m(N)}\}$,
with $0 \le m(1) < m(2) < \cdots < m(N)$. Let
$$
h := \prod_{g \in \mathcal F_0} g = \prod_{j=1}^N f_{m(j)}.
$$
Tri\-via\-l\-ly, $h \in HC(C_\varphi )$
if $N=1$. If $N > 1$ then $h = h_0 \cdot f_{m(N)}$, where $h_0 := \prod_{j=1}^{N-1} f_{m(j)}$.
Fix \,$F \in H(G)$.
As \,$f_{m(N)} \in \mathcal U ((C_\varphi^{p(m(N),k)}))$, there is a subsequence $(\nu_k ) \subset (p(m(N),k))$ \,such that \,$f_{m(N)} \circ \varphi^{\nu_k} \to F$ $(k \to \infty )$
uniformly on compacta in $G$. But \,$f_{m(j)} \circ \varphi^{\nu_k} \to 1$ $(k \to \infty )$ \,uniformly on compacta in \,$G$ \,for each
\,$j \in \{1,\dots ,N-1\}$.
This entails \,$h_0 \circ \varphi^{\nu_k} \to 1$ \,in \,$H(G)$. Thus
$$
h \circ \varphi^{\nu_k} = (h_0 \circ \varphi^{\nu_k}) \cdot (f_{m(N)} \circ \varphi^{\nu_k}) \longrightarrow 1 \cdot F = F \hbox{ \ in \ } H(G).
$$
Consequently, $h \in HC(C_\varphi )$. So far, (a), (b) and (c) have been proved.

\vskip 5pt

In order to prove (d), assume, by way of contradiction, that there is some $N$-tuple $(c_0,c_1, \dots ,c_N) \in \C^{N+1} \setminus \{(0, \dots , 0)\}$
such that $c_0f_0 + \dots + c_Nf_N = 0$, where, without loss of generality, we can suppose $c_N \ne 0$. Then
$c_0 (f_0 \circ \varphi^{p(N,k)}) + \dots + c_N (f_N \circ \varphi^{p(N,k)}) = 0$ for all $k \ge 1$. Letting $k \to \infty$ we get
$c_0 \cdot 1 + \cdots + c_{N-1} \cdot 1 + c_N (f_N \circ \varphi^{p(N,k)})  \to 0$, so
$$
f_N \circ \varphi^{p(N,k)} \longrightarrow - \sum_{j=0}^{N-1} c_j c_N^{-1} \hbox{ \ in \ } \mathcal H (G) \hbox{ \ as \ } k \to \infty,
$$
which is absurd because $\{f_N \circ \varphi^{p(N,k)}\}_{k \ge 1}$ is dense in \,$\mathcal H (G)$. The theorem is now totally proved.
\end{proof}

\noindent{\bf Acknowledgments.} 

The authors are indebted to the referee for helpful comments and suggestions, specially for nice simplifications of the proofs of some results.

The first author has been supported by the Plan
Andaluz de Investigaci\'on de la Junta de Andaluc\'{\i}a FQM-127 Grant P08-FQM-03543 and by MEC Grant MTM2015-65242-C2-1-P.
The second author has been supported by MEC Grant MTM2016-75963-P.
The second and third authors were also supported by Generalitat Valenciana, Project GV/2010/091, and by Universitat Polit\`ecnica de Val\`encia, Project PAID-06-09-2932. The fourth author has been supported by Grant MTM2015-65825-P.


\end{document}